\documentclass[recqno]{amsproc}

\usepackage[margin=1in]{geometry}
\usepackage{setspace, fullpage}
\usepackage{amsmath,amsthm,amscd,amssymb,amsbsy, bbm,wasysym,mathrsfs}
\usepackage{latexsym}
\usepackage[colorlinks,citecolor=blue,pagebackref,hypertexnames=false]{hyperref}
\numberwithin{equation}{section}
\theoremstyle{plain}
\newtheorem{theorem}{{\bf Theorem}}[section]
\newtheorem{lemma}[theorem]{{\bf Lemma}}
\newtheorem{corollary}[theorem]{Corollary}
\newtheorem{proposition}[theorem]{Proposition}

\theoremstyle{definition}
\newtheorem{definition}[theorem]{{\bf Definition}}

\theoremstyle{remark}
\newtheorem{remark}[theorem]{Remark}
\numberwithin{equation}{section}

\begin{document}

\title{Roundness properties of ultrametric spaces}

\author{Timothy Faver}
\address{Department of Mathematics, Drexel University, Philadelphia, PA 19104, USA}
\email{tef@drexel.edu}
\author{Katelynn Kochalski}
\address{Department of Mathematics, University of Virginia, Charlottesville, VA 22904, USA}
\email{kdk7rn@virginia.edu}
\author{Mathav Kishore Murugan}
\address{Center for Applied Mathematics, Cornell University, Ithaca, NY 14853, USA}
\email{mkm233@cornell.edu}
\author{Heidi Verheggen}
\address{Department of Economics, Cornell University, Ithaca, NY 14853, USA}
\email{hv59@cornell.edu}
\author{Elizabeth Wesson}
\address{Center for Applied Mathematics, Cornell University, Ithaca, NY 14853, USA}
\email{enw27@cornell.edu}
\author{Anthony Weston}
\address{Department of Mathematics and Statistics, Canisius College, Buffalo, NY 14208, USA}
\address{Department of Decision Sciences, University of South Africa, Pretoria, Gauteng 0003, South Africa}
\email{westona@canisius.edu}



\begin{abstract}
Motivated by a classical theorem of Schoenberg we prove that an $n + 1$ point finite metric space
has strict $2$-negative type if and only if it can be isometrically embedded in the Euclidean
space $\mathbb{R}^{n}$ of dimension $n$ but it cannot be isometrically embedded in any Euclidean
space $\mathbb{R}^{r}$ of dimension $r < n$. We use this result as a technical tool to
study ``roundness'' properties of additive metrics with a particular focus
on ultrametrics and leaf metrics. The following conditions are shown to be equivalent for a
metric space $(X,d)$: (1) $X$ is ultrametric, (2) $X$ has infinite roundness, (3) $X$ has infinite
generalized roundness, (4) $X$ has strict $p$-negative type for all $p \geq 0$, and (5) $X$ admits
no $p$-polygonal equality for any $p \geq 0$. As all ultrametric spaces have strict $2$-negative
by (4) we thus obtain a short new proof of Lemin's theorem: Every $n + 1$ point ultrametric space is
isometrically embeddable into some Euclidean space as an affinely independent set. Motivated by a
question of Lemin, Shkarin introduced
the class $\mathcal{M}$ of all finite metric spaces that may be isometrically embedded into $\ell_{2}$
as an affinely independent set. The results of this paper show that Shkarin's class $\mathcal{M}$
consists of all finite metric spaces of strict $2$-negative type. We also note that it is possible
to construct an additive metric space whose generalized roundness is exactly $\wp$ for each $\wp \in [1, \infty]$.
\end{abstract}
\maketitle

\section{Introduction}\label{s1}
The purpose of this paper is to study roundness properties of additive metrics with a particular focus
on ultrametrics and leaf metrics. All preliminary material
will be presented in Section \ref{s2}. Section \ref{s3} examines the precise relationship between
strict $2$-negative type and isometric Euclidean embeddings in the context of finite metric spaces.
A number of applications to discrete and distance geometry are then derived from this relationhsip.
Section \ref{s4} examines the range of possible values of the
generalized roundness of additive metric spaces. In the case of finite additive metrics that are not
ultrametrics, this range is seen to be $(1, \infty)$. Examples of finite leaf metrics are given to show
that all values in this range are in fact attained. The main results in Section \ref{s5} show that the following
conditions are equivalent for a metric space $(X,d)$:
\begin{enumerate}
\item $X$ is ultrametric.

\item $X$ has infinite roundness.

\item $X$ has infinite generalized roundness.

\item $X$ has strict $p$-negative type for all $p \geq 0$.

\item $X$ admits no $p$-polygonal equality for any $p \geq 0$.
\end{enumerate}
It is possible to draw a large number of corollaries from these equivalences. For example, on the basis of 
the equivalence of (1) and (4), it is possible to both recover and generalize well-known
theorems concerning the isometric embedding of ultrametric spaces into Hilbert spaces. We do this by showing
that if $n > 1$ is an integer and if $(X,d)$ is an $n + 1$ point metric space, then $(X,d)$ has
strict $2$-negative type if and only if $(X,d)$ can be isometrically
embedded in the Euclidean space $\mathbb{R}^{n}$ of dimension $n$ but it cannot be isometrically
embedded in any Euclidean space $\mathbb{R}^{r}$ of dimension $r< n$.
As all ultrametric spaces have strict $2$-negative type by (4), this generalizes Lemin \cite[Theorem 1.1]{Lem}.
More generally, we note that if an infinite metric space $(X,d)$ of cardinality $\psi$ has $p$-negative type for some $p > 2$,
then $(X,d)$ can be isometrically embedded as a closed subset in a Euclidean space of algebraic dimension $\psi$
but it cannot be isometrically embedded in any Euclidean space of algebraic dimension $\sigma < \psi$.
This generalizes Lemin \cite[Theorem 1.2]{Lem}.

\section{Ultrametrics, additive metrics, leaf metrics and roundness}\label{s2}
We begin by recalling the definition of an ultrametric space.

\begin{definition}\label{ultra}
A metric $d$ on a set $X$ is said to be \textit{ultrametric} if for all $x,y,z \in X$, we have:
\[
d(x,y) \leq \max \{ d(x,z), d(y,z) \}.
\]
\end{definition}

An important basic property of ultrametrics is expressed by the following well-known lemma.

\begin{lemma}\label{c1}
If $(X,d)$ is an ultrametric space, then $(X,d^{p})$ is an ultrametric space for all $p \geq 0$.
\end{lemma}

\begin{proof}
Let $p \geq 0$ and points $x,y,z \in X$ be given. By assumption, $d(x,z) \leq \max \{d(x,y), d(y,z) \}$.
Clearly, $d(x,z)^{p} \leq \{ \max \{d(x,y), d(y,z) \} \}^{p} = \max \{ d(x,y)^{p}, d(y,z)^{p} \}$. Thus
$d^{p}$ is an ultrametric on $X$.
\end{proof}

Interesting examples of ultrametric spaces include the rings $Z_{p}$ of $p$-adic integers, the Baire space
$B_{\aleph_{0}}$, non-Archimedean normed fields and rings of meromorphic functions on open regions of the complex plane.
There is an immense literature surrounding ultrametrics as they have been intensively studied
by topologists, analysts, number theorists and theoretical biologists for the best part of the last $100$ years.
For example, de Groot \cite{deG} characterized ultrametric spaces up to homeomorphism as the strongly
zero-dimensional metric spaces. In numerical taxonomy, on the other hand, every finite ultrametric space is known
to admit a natural hierarchical description called a \textit{dendogram}.
This has significant ramifications in theoretical biology. See, for instance, \cite{Gor}.
More recently, ultrametrics have figured prominently in the study of embeddings of finite metric spaces into
various ambient spaces such as $\ell_{1}$ and $\ell_{2}$. There are many interesting papers along these lines, including
\cite{Bar, Lin, Fak}. Finally, as noted in \cite{Wat}, ultrametric spaces play a significant role in
other fields such as statistical mechanics, neural networks and combinatorial optimization.

In fact, ultrametrics are special instances of a more general class of metrics which are termed additive. As we have
noted in Definition \ref{ultra}, ultrametrics are defined by a stringent three point criterion. The class of additive metrics
satisfy a more relaxed four point criterion. The formal definition is as follows.

\begin{definition}
A metric $d$ on a set $X$ is said to be \textit{additive} if for all $x,y,z,w$ in $X$, we have:
\[
d(x,y) + d(z,w) \leq \max \{ d(x,z) + d(y,w), d(x,w) + d(y,z) \}.
\]
\end{definition}

Recall that a \textit{metric tree} is a connected graph $(T,E)$ without cycles or loops in which each edge $e \in E$
is assigned a positive length $|e|$. The distance $d_{T}(x,y)$ between any two vertices $x,y \in T$ is then defined
to be the sum of the lengths of the edges that make up the unique minimal geodesic from $x$ to $y$.
The tree is said to be \textit{discrete} if there is a positive constant $c$ such that $|e| > c$ for all $e \in E$.
It is an easy exercise to verify that every metric tree is additive. Theorem 2 in \cite{Bun}, which is known as
\textit{Buneman's criterion}, shows that a finite metric space is additive if and only if it is a tree metric in the
sense of the following definition.

\begin{definition}
A metric $d$ on a set $X$ is said to be a \textit{tree metric} if there exists a finite metric tree $(T,E,d_{T})$
such that:
\begin{enumerate}
\item $X$ is contained in the vertex set $T$ of the tree, and

\item $d(x, y) = d_{T} (x, y)$ for all $x,y \in X$.
\end{enumerate}
In other words, $d$ is a tree metric if $(X,d)$ is isometric to a metric subspace of some metric tree.
\end{definition}

The notions of finite metric tree and tree metric are distinct. The natural graph metric on a finite
tree is a tree metric but a tree metric $d$ on a finite set $X$ does not necessarily imply the existence
of an edge set that makes $(X,d)$ into a finite metric tree.

Ultrametrics form a very special sub-class of the collection of all additive metrics. Indeed,
there is a close relationship between ultrametric spaces and the leaf sets or end spaces of certain trees.
This type of identification is discussed more formally in \cite{Mic, Fie, Hol}. However, metrics that arise by
restricting the (possibly weighted) graph metric of a tree to the leaf set of the tree are more general
than ultrametrics, yet still form an interesting sub-class of the collection of all additive metrics.
It is helpful to extract the following definition.

\begin{definition}
A metric $d$ on a set $X$ is a \textit{leaf metric} if there exists a discrete metric tree $(T,E,d_{T})$
such that:
\begin{enumerate}
\item $X$ is contained in the set $\mathscr{L}_{T}$ of all leaves of $T$, and 

\item $d(x, y) = d_{T} (x, y)$ for all $x,y \in X$.
\end{enumerate}
\end{definition}

In the case of finite metrics the following proper inclusions are easily seen to hold:
\[
\{\mbox{ultrametrics}\} \subset \{\mbox{leaf metrics}\} \subset \{\mbox{additive metrics}\}.
\]
It is possible to extend this hierarchy to the infinite case in certain ways but we will not go into the specifics
here, deferring instead to \cite{Gor, Hol, Hu1, Hu2}.

The focus of this paper will be on roundness properties of ultrametrics, leaf metrics and additive metrics.
The notion of the roundness of a metric space was introduced by Enflo \cite{En1, En3, En4} in order to study
the uniform structure of $L_{p}$-spaces ($0 < p \leq 2$).

\begin{definition}\label{roundness} Let $p \geq 1$ and let $(X,d)$ be a metric space. We say that $p$ is a
\textit{roundness exponent} of $(X,d)$, denoted by $p \in r(X)$ or $p \in r(X,d)$, if and only if for all
quadruples of points $x_{00}, x_{01}, x_{11}, x_{10} \in X$, we have:
\begin{eqnarray}\label{r1}
d(x_{00},x_{11})^{p} + d(x_{01},x_{10})^{p} & \leq & d(x_{00},x_{01})^{p} + d(x_{01},x_{11})^{p} +
d(x_{11},x_{10})^{p} + d(x_{10},x_{00})^{p}.
\end{eqnarray}
The \textit{roundness} of $(X,d)$, denoted by $\mathfrak{p}(X)$ or $\mathfrak{p}(X,d)$, is then
defined to be the supremum of the set of all roundness exponents of $(X,d)$: $\mathfrak{p}(X,d) = \sup \{ p : p \in r(X,d) \}$.
\end{definition}

The restriction $p \geq 1$ in the statement of Definition \ref{roundness} is not necessary.
However, $p = 1$ is roundness exponent of any given metric space $(X,d)$ by the triangle inequality.
Thus $\mathfrak{p}(X,d) \geq 1$. The following simple proposition points out that every ultrametric space has
infinite roundness.

\begin{proposition}\label{roundexp}
The set of roundness exponents of an ultrametric space is the interval $[1, \infty)$.
\end{proposition}

\begin{proof}
Let $(X,d)$ be an ultrametric space and let $p \geq 1$ be given. By Lemma \ref{c1}, $(X, d^{p})$
is a metric space. Thus $1 \in R(X, d^{p})$, and so $p$ is a roundness exponent of $(X,d)$.
\end{proof}

Notions of negative type and generalized roundness were formally introduced and studied by
Menger \cite{Men}, Schoenberg \cite{Sc1, Sc2, Sc3} and Enflo \cite{En2} respectively.
Rather surprisingly, the notions of negative type and generalized roundness are actually equivalent.
This equivalence is recalled in Theorem \ref{LTWthm}.
More recently, there has been interest in the notion of strict $p$-negative type, particularly
as it pertains to the geometry of finite metric spaces. Papers which have been instrumental
in developing a basic theory of strict $p$-negative type metrics include \cite{Hj1, Hj2, Dou, Hli, Nic, Wol, San}.

\begin{definition}\label{NTGRDEF} Let $p \geq 0$ and let $(X,d)$ be a metric space.
\begin{enumerate}
\item $(X,d)$ has $p$-{\textit{negative type}} if and only if for
all finite subsets $\{x_{1}, \ldots , x_{n} \} \subseteq X$ and all choices of real numbers $\eta_{1},
\ldots, \eta_{n}$ with $\eta_{1} + \cdots + \eta_{n} = 0$, we have:
\begin{eqnarray}\label{pneg}
\sum\limits_{1 \leq i,j \leq n} d(x_{i},x_{j})^{p} \eta_{i} \eta_{j} & \leq & 0.
\end{eqnarray}

\item $(X,d)$ has \textit{strict} $p$-{\textit{negative type}} if and only if it has $p$-negative type
and the inequality (\ref{pneg}) is strict whenever the scalar $n$-tuple $(\eta_{1}, \ldots , \eta_{n}) \not= (0, \ldots, 0)$.

\item We say that $p$ is a \textit{generalized roundness exponent} for $(X,d)$, denoted by $p \in gr(X)$
or $p \in gr(X,d)$, if and only if for all integers $n > 0$,
and all choices of points $a_{1}, \ldots , a_{n}, b_{1}, \ldots , b_{n} \in X$, we have:
\begin{eqnarray}\label{grinq}
\sum\limits_{1 \leq k < l \leq n} \left\{ d(a_{k},a_{l})^{p} + d(b_{k},b_{l})^{p} \right\} & \leq &
\sum\limits_{1 \leq j,i \leq n} d(a_{j},b_{i})^{p}.
\end{eqnarray}

\item The \textit{generalized roundness} of $(X,d)$, denoted by $\mathfrak{q}(X)$ or
$\mathfrak{q}(X,d)$, is defined to be the supremum of the set of all generalized roundness exponents
of $(X,d)$: $\mathfrak{q}(X,d)  =  \sup \{ p : p \in gr(X,d) \}$.
\end{enumerate}
\end{definition}

\begin{remark}\label{int:rem}
The set of all roundness exponents of a general metric space $(X,d)$ need not
form an interval. Indeed, Enflo \cite{En4} constructed a four point metric space $(X,d)$ such that
$2 \notin r(X,d)$ but $q \in r(X,d)$ for some $q > 2$. However, in the case of a Banach space $(X, \| \cdot \|)$,
$r(X)$ is a (possibly degenerate) closed interval of the form $[1,\wp]$ for some $\wp \in [1,2]$.
This result is Proposition 4.1.2 in \cite{En4}. By way of comparison, the set of generalized roundness exponents
of a general metric space $(X,d)$ always forms either a (possibly degenerate) closed interval of the form
$[0,\wp]$ for some $\wp \in [0,\infty)$ or a closed interval of the form $[0,\infty)$. This follows
from Theorem 2 in \cite{Sc2} and Theorem \ref{LTWthm}. (See also Corollary 2.5 in
Lennard \textit{et al}.\ \cite{Ltw}.)
\end{remark}

It is plainly evident that every roundness inequality (\ref{r1}) can be expressed in the form
of a generalized roundness inequality (\ref{grinq}) with the same exponent. Hence, for any given metric
space $(X,d)$, generalized roundness cannot exceed roundness: $\mathfrak{q}(X,d) \leq \mathfrak{p}(X,d)$.
This fact together with two other basic properties of generalized roundness are recorded in the following
lemma. These results are well-known and easily verified on the basis of the definitions.

\begin{lemma}\label{c2}
Let $(X,d)$ be a metric space. Then the following statements hold:
\begin{enumerate}
\item[(1)] $\mathfrak{q}(X,d) \le \mathfrak{p}(X,d)$.

\item[(2)] If $\mathfrak{q}(X,d) < \infty$, then $\mathfrak{q}(X,d) \in \mbox{gr}(X,d)$.

\item[(3)] If $p > 0$ is such that $d^p$ is a metric on $X$ and if $\beta \in \mbox{gr}(X, d^p)$,
then $\beta p \in \mbox{gr}(X,d)$.
\end{enumerate}
\end{lemma}

As noted earlier, a surprising fact is that conditions (1) and (3) of Definition \ref{NTGRDEF} are actually equivalent.

\begin{theorem}\label{LTWthm}
Let $p \geq 0$ and let $(X,d)$ be a metric space. Then the following conditions are equivalent:
\begin{enumerate}
\item $(X,d)$ has $p$-negative type.

\item For all $s, t \in \mathbb{N}$, all choices of pairwise distinct points $a_{1}, \ldots, a_{s},
b_{1}, \ldots, b_{t} \in X$ and all choices of real numbers $m_{1}, \ldots, m_{s}, n_{1}, \ldots, n_{t} > 0$
such that $m_{1} + \cdots + m_{s} = n_{1} + \cdots + n_{t}$, we have:
\begin{eqnarray}\label{wgr}
\sum\limits_{1 \leq j_{1} < j_{2} \leq s} m_{j_{1}}m_{j_{2}}d(a_{j_{1}},a_{j_{2}})^{p}\,\, +
\sum\limits_{1 \leq i_{1} < i_{2} \leq t} n_{i_{1}}n_{i_{2}}d(b_{i_{1}},b_{i_{2}})^{p}
& \leq & \sum\limits_{j,i=1}^{s,t} m_{j}n_{i}d(a_{j},b_{i})^{p}.
\end{eqnarray}

\item $p$ is a generalized roundness exponent for $(X,d)$.
\end{enumerate}
Moreover, $(X,d)$ has strict $p$-negative type if and only if the inequality (\ref{wgr}) is strict
for all choices of pairwise distinct points $a_{1}, \ldots, a_{s},                                    
b_{1}, \ldots, b_{t} \in X$ and all choices of real numbers $m_{1}, \ldots, m_{s}, n_{1}, \ldots, n_{t} > 0$
such that $m_{1} + \cdots + m_{s} = n_{1} + \cdots + n_{t}$.
\end{theorem}

The equivalence of conditions (1) --- (3) in Theorem \ref{LTWthm} is due to Lennard \textit{et al}.\ \cite[Theorem 2.4]{Ltw}.
The statement concerning strict $p$-negative type is due to Doust and Weston \cite[Remark 2.5]{Dou}. An inequality
of the form (\ref{wgr}) that holds at equality will be called a \textit{$p$-polygonal equality}.
The importance of $p$-polygonal equalities is that they correspond in a very precise way to
the non-trivial instances of equality in (\ref{pneg}). Non-trivial in this context means that not all
$\eta_{1}, \ldots, \eta_{n}$ are $0$ in (\ref{pneg}). The transition between the two types of equality is
described in several places, including Weston \cite{We2}. In the case of a finite metric space $(X,d)$,
if $\wp = \mathfrak{q}(X,d) < \infty$, then $X$ admits a $\wp$-polygonal equality. This follows directly
from Corollary 4.4 and Theorem 5.4 in Li and Weston \cite{Hli}.

\section{Euclidean embeddings and strict negative type}\label{s3}
In this section we examine the interplay between
strict $p$-negative type and isometric Euclidean embeddings in the context
of finite metric spaces. This leads to interesting variants of results of Schoenberg \cite{Sc1},
Deza and Maehara \cite{Dez}, Nickolas and Wolf \cite{Nic}, Lemin \cite{Lem}, and Shkarin \cite{Shk}.

Given a finite metric space $(X,d) = (\{ x_{0}, x_{1}, \ldots, x_{n} \}, d)$ and a non-negative real number $p$ we let $D_{p}$
denote the associated $(n+1) \times (n+1)$ \textit{$p$-distance} matrix $(d^{p}(x_{j},x_{k}))$
($0 \leq j,k \leq n$) and we let $A_{p} = (a_{jk})$ denote the symmetric $n \times n$ matrix with entries
\[
a_{jk} =  \frac{1}{2} \biggl( d^{p}(x_{0}, x_{j}) + d^{p}(x_{0}, x_{k}) - d^{p}(x_{j}, x_{k}) \biggl), 1 \leq j,k \leq n.
\]
The following fundamental result of Schoenberg (which appears as Theorem 1 in \cite{Sc1}) shows that properties of the matrix
$A_{2}$ determine if $(X,d)$ can be isometrically embedded in a Euclidean space of finite dimension.

\begin{theorem}[Schoenberg \cite{Sc1}]\label{ScThm1}
Let $n > 1$ and $m \leq n$ be natural numbers. An $n + 1$ point metric space
$(X,d) = (\{ x_{0}, x_{1}, \ldots, x_{n} \}, d)$ can be isometrically embedded
in the Euclidean space $\mathbb{R}^{m}$ of dimension $m$ and no smaller dimension if and only if
the matrix $A_{2}$ is positive semi-definite and of rank $m$.
\end{theorem}

With these ideas in mind we now recast the notion of strict $p$-negative type for finite metric spaces.

\begin{lemma}\label{posdef}
Let $n > 1$ be an integer and let $(X,d) = (\{ x_{0}, x_{1}, \ldots, x_{n} \}, d)$ be an $n+1$ point metric space.
For each real number $p > 0$ the following conditions are equivalent:
\begin{enumerate}
\item $(X,d)$ has strict $p$-negative type.

\item The matrix $A_{p}$ is positive definite.
\end{enumerate}
\end{lemma}

\begin{proof}
Let $p > 0$ be given.

$(1) \Rightarrow (2)$ Suppose $(X,d)$ has strict $p$-negative type. If for each vector $\boldsymbol{\eta} =
(\eta_{1}, \ldots, \eta_{n})^{T} \in \mathbb{R}^{n}$ we set $\eta_{0} = - (\eta_{1} + \cdots + \eta_{n})$
and define $\boldsymbol{\eta}_{\ast} = (\eta_{0}, \eta_{1}, \ldots, \eta_{n})^{T} \in \mathbb{R}^{n+1}$, then
it is not difficult to verify that $2 (\boldsymbol{\eta}^{T} \cdot A_{p} \cdot \boldsymbol{\eta})
= - \boldsymbol{\eta}^{T}_{\ast} \cdot D_{p} \cdot \boldsymbol{\eta}_{\ast}$. By definition of strict $p$-negative type,
$\boldsymbol{\eta}^{T}_{\ast} \cdot D_{p} \cdot \boldsymbol{\eta}_{\ast} < 0$ for all non-trivial vectors
$\boldsymbol{\eta}_{\ast} \in \mathbb{R}^{n+1}$ whose coordinates sum to $0$. This implies that
$\boldsymbol{\eta}^{T} \cdot A_{p} \cdot \boldsymbol{\eta} = - \frac{1}{2} \boldsymbol{\eta}^{T}_{\ast}
\cdot D_{p} \cdot \boldsymbol{\eta}_{\ast}) > 0$ for all non-trivial vectors $\boldsymbol{\eta} \in \mathbb{R}^{n}$.
In other words, the matrix $A_{p}$ is positive definite.

$(2) \Rightarrow (1)$ Suppose that $A_{p}$ is positive definite.
Consider a non-trivial vector $\boldsymbol{\eta}_{\ast} = (\eta_{0}, \eta_{1}, \ldots,
\eta_{n})^{T}$ in $\mathbb{R}^{n+1}$ whose coordinates sum to $0$. Then it follows that the vector
$\boldsymbol{\eta} = (\eta_{1}, \ldots, \eta_{n})^{T} \in \mathbb{R}^{n}$ is also non-trivial and that
$\eta_{0} = - (\eta_{1} + \cdots + \eta_{n})$. As before, we see that
$- \boldsymbol{\eta}^{T}_{\ast} \cdot D_{p} \cdot \boldsymbol{\eta}_{\ast}
= 2 (\boldsymbol{\eta}^{T} \cdot A_{p} \cdot \boldsymbol{\eta}) > 0$, which shows that $(X,d)$ has strict $p$-negative type.
\end{proof}

Results of Li and Weston \cite{Hli} show that a finite metric space has strict $p$-negative type
if and only if it has $q$-negative type for some $q > p$. An important consequence of Lemma \ref{posdef}
is that it provides for a version of Theorem \ref{ScThm1} that is predicated in terms of strict $2$-negative type.

\begin{theorem}\label{embed}
Let $n > 1$ be an integer. For an $n + 1$ point metric space $(X,d) = (\{ x_{0}, x_{1}, \ldots, x_{n} \}, d)$,
the following conditions are equivalent:
\begin{enumerate}
\item[(1)] $(X,d)$ has strict $2$-negative type.

\item[(2)] $(X,d)$ can be isometrically embedded in the Euclidean space $\mathbb{R}^{n}$ of dimension
$n$ but it cannot be isometrically embedded in any Euclidean space $\mathbb{R}^{r}$ of dimension $r< n$.

\item[(3)] There exists an isometry $\phi: (X, d) \rightarrow \mathbb{R}^{n}$ such that the set
$$\{ \phi(x_{1}) - \phi(x_{0}), \phi(x_{2}) - \phi(x_{0}), \ldots, \phi(x_{n}) - \phi(x_{0})\}$$
is linearly independent.
\end{enumerate}
\end{theorem}

\begin{proof} Let $n > 1$ be an integer and let $(X,d) = (\{ x_{0}, x_{1}, \ldots, x_{n} \}, d)$
be an $n+1$ point metric space. We will use the same notation that was employed in the statement
and proof of Lemma \ref{posdef}.

(1) $\Rightarrow$ (2) Suppose that $(X,d)$ has strict $2$-negative type. By Lemma \ref{posdef}, the
matrix $A_{2}$ is positive definite, and so it must be of rank $n$. Therefore $(X,d)$ can be isometrically
embedded in the Euclidean space $\mathbb{R}^{n}$ of dimension $n$ but it cannot be isometrically
embedded in any Euclidean space $\mathbb{R}^{r}$ of dimension $r< n$ by Theorem \ref{ScThm1}.
(See also Theorem 2.1, Corollary 2.2 and Theorem 2.4 in Wells and Williams \cite{Wel}.)

(2) $\Rightarrow$ (1) Suppose $(X,d)$ can be isometrically embedded in the Euclidean space $\mathbb{R}^{n}$
of dimension $n$ but it cannot be isometrically embedded in any Euclidean space $\mathbb{R}^{r}$ of
dimension $r< n$. By Theorem \ref{ScThm1}, the matrix $A_{2}$ is positive semi-definite and of full rank.
So, in fact, $A_{2}$ is positive definite. Thus $(X,d)$ has strict $2$-negative type by Lemma \ref{posdef}.

The equivalence of (2) and (3) is plain.
\end{proof}

The condition (3) in the statement of Theorem \ref{embed} may be restated in the following more elegant way:
$(X,d)$ is isometrically embeddable into some Euclidean space as an affinely independent set.
Thus, bearing in mind that the transform $d^{p/2}$ is a metric on $X$ provided $0 \leq p \leq 2$, we
obtain the following immediate corollary of Theorem \ref{embed}.

\begin{corollary}\label{cor3.5}
Let $0 \leq p \leq 2$. A finite metric space $(X,d)$ has strict $p$-negative type if and only if
$(X,d^{p/2})$ is isometrically embeddable into some Euclidean space as an affinely independent set.
\end{corollary}

It is possible to use Corollary \ref{cor3.5} in conjunction with existing theory to identify
large classes of finite metric spaces that have strict $p$-negative type for a given $p$, $0 < p \leq 2$.
For example, the next corollary says that almost equilateral simplices necessarily have strict $2$-negative type.

\begin{corollary}\label{cor3.75}
Let $n > 1$ be an integer. If $(X,d) = (\{ x_{0}, x_{1}, \ldots, x_{n} \}, d)$ is an $n+1$ point
metric space with all non-zero distances satisfying $1 \leq d(x_{j}, x_{i}) \leq \gamma_{n}$ where
\[
\gamma_{n} =
\begin{cases}
\sqrt{1 + \frac{2n + 1}{n^{2} - 2}} & \text{if $n$ is even} \\
\sqrt{1 + \frac{2}{n - 1}} & \text{if $n$ is odd,}
\end{cases}
\]
then $(X,d)$ has strict $2$-negative type.
\end{corollary}

\begin{proof}
Dekster and Wilker \cite{Dek} have shown that any such finite metric space may isometrically embedded
into $\ell_{2}^{(n)}$ as an affinely independent set. The result now follows from Corollary \ref{cor3.5}
in the case $p = 2$.
\end{proof}

Every finite metric space has $p$-negative type for some $p > 0$. In fact, if $(X,d)$ is an
$n + 1$ point metric space for some integer $n > 1$, then there is a positive constant
$\wp = \wp(n) > 0$, depending only on $n$, such that $(X,d)$ has $\wp$-negative type.
This result was obtained independently by Deza and Maehara \cite{Dez} (Theorem 3) and
Weston \cite{We1} (Theorem 4.3). Indeed, it can be inferred from \cite{Dez} that we may take
$\wp = \log_{2}(1 + 1/n)$. Moreover,
as strict negative type holds on intervals by Theorem 5.4 in Li and Weston \cite{Hli},
we may thus deduce the following theorem from Theorem \ref{embed} and Corollary \ref{cor3.5}.

\begin{theorem}\label{thm3.75}
Let $n > 1$ be an integer and let $\wp = \log_{2}(1 + 1/n)$. If $(X,d)$ is an $n + 1$ point metric
space and if $0 \leq p < \wp$, then the metric space $(X, d^{p/2})$ can be isometrically embedded
in the Euclidean space $\mathbb{R}^{n}$ of dimension $n$ but it cannot be isometrically embedded
in any Euclidean space $\mathbb{R}^{r}$ of dimension $r< n$.
\end{theorem}

\begin{remark}
Theorem \ref{thm3.75} is a version of Corollary 3 in Deza and Maehara \cite{Dez}
but with additional control over the nature of the isometry. For technical reasons which we
shall not elaborate here, there is no expectation that the value
$\wp(n) = \log_{2}(1 + 1/n)$ used in the statement of Theorem \ref{thm3.75} is optimal.
It is a challenging open problem to compute $\sup \wp$ where the supremum is taken over all $\wp = \wp(n) > 0$
for which Theorem \ref{thm3.75} holds. This relates to the conjecture in Section $5$ of \cite{Dez}.
In the restricted case of an $n + 1$ point metric tree endowed with the usual graph metric, Theorem
\ref{thm3.75} will necessarily hold for a value of $\wp = \wp(n) > 1$ that only depends on $n$. This follows
from Corollary \ref{cor3.5} and Corollary 5.5 in Doust and Weston \cite{Dou}. The best known value of $\wp(n)$
for unweighted $n + 1$ point metric trees is provided by Corollary 3.4 in \cite{Hli}.
Another important aspect of Corollary \ref{cor3.5} and Theorem \ref{thm3.75} arises in situations
where $p = 1$. This is the case of strict $1$-negative type. For example, we obtain a version
of Theorem 3.4 (3) in Nickolas and Wolf \cite{Nic} by setting $p = 1$ in Corollary \ref{cor3.5}.
However, Theorem 3.4 (3) in \cite{Nic} provides additional measure theoretic information.
\end{remark}

There is a neat way to apply Theorem \ref{thm3.75} to one of the fundamental problems of distance geometry;
namely, the problem of realizing graphs in Euclidean spaces.
Let $G = (V, E)$ be a finite simple connected graph whose edges are weighted by some function $d: E \rightarrow (0, \infty)$
and let $k > 1$ be an integer. Recall that, by definition, a \textit{realization of the graph $G$ in $\ell_{2}^{(k)}$}
is a map $\varkappa : V \rightarrow \ell_{2}^{(k)}$ such that
\begin{eqnarray}\label{dg:prob}
\forall\, \{ u,v \} \in E,\,\,\, d(\{ u,v \}) & = & \| \varkappa(u) - \varkappa(v) \|_{2}.
\end{eqnarray}
It is worth stressing that such a function $\varkappa$ is not an embedding in any strict topological sense.
The notion of graph realization is only predicated upon the lengths of the edges $\{ u,v \} \in E$.
Now assume that $n = |G| > 2$ and suppose that $d(\{ u,v \}) = 1$ for each $\{ u,v \} \in E$. By Theorem \ref{thm3.75},
there is a $p > 0$ such that $(G, d^{p/2})$ can be isometrically embedded into $\ell_{2}^{(n-1)}$ as an affinely
independent set. In particular, as $1^{p/2} = 1$, we see immediately that $G$ may be realized in $\ell_{2}^{(n-1)}$.
In summary, we have obtained the following result.

\begin{corollary}
Let $G$ be a finite simple connected graph endowed with the ordinary graph metric and suppose that $n = |G| > 2$.
Then the graph $G$ may be realized as an affinely independent set in $\ell_{2}^{(n-1)}$.
\end{corollary}

It is interesting to apply the theory discussed in this section to certain Cayley graphs.
The coarse Baum-Connes conjecture postulates an algorithm for computing the higher indices of
generalized elliptic operators on non-compact spaces \cite{Hig}.
Any finitely generated group that admits a Cayley graph of positive generalized roundness satisfies the coarse
Baum-Connes conjecture (and thus the strong Novikov conjecture) by a result of Lafont and Prassidis \cite{Laf}.
And any finite or additive Cayley graph has positive generalized roundness by \cite{Dez, We1} or by
Proposition \ref{addone} (in the next section). These considerations lead to the following corollary.

\begin{corollary}
Every finitely generated group $\Gamma$ that admits a finite or additive Cayley graph satisfies the
coarse Baum-Connes conjecture.
\end{corollary}

We complete this section with one final observation about affinely independent Euclidean embeddings.
In the case of finite metric spaces there is another way to think about strict $2$-negative type.
According to Theorem \ref{embed}, an $n + 1$ point metric space $(X,d) = (\{ x_{0}, x_{1}, \ldots , x_{n} \}, d)$
has strict $2$-negative type if and only if it can be isometrically embedded into $\ell_{2}^{(n)}$ as an affinely
independent set. The latter condition has been studied classically and may be characterized in terms of Cayley-Menger
determinants. The theory of Cayley-Menger determinants is due to Menger \cite{Me1}. (See also Blumenthal \cite{Bl2}.)
Given that $D_{2} = (d^{2}(x_{i},x_{j}))$ denotes the $2$-distance matrix of $(X,d)$,
the \textit{Cayley-Menger determinant} of $(X,d)$ is defined as follows:
\[
\text{CM}\det(\{x_{0}, \ldots, x_{n}\}, d) =
\left|
\begin{matrix}
 & & & 1 \\
 & D_{2} & & \vdots \\
 & & & 1 \\
1 & \cdots & 1 & 0
\end{matrix}
\right|
\]

Menger \cite{Me1} used these determinants to characterize condition (3) of Theorem \ref{embed} as follows.

\begin{theorem}[Menger \cite{Me1}]\label{menger}
Let $(X,d) = (\{ x_{0}, x_{1}, \ldots , x_{n} \}, d)$ be an $n + 1$ point metric space $(n > 1)$.
Then the following conditions are equivalent:
\begin{enumerate}
\item[(1)] $(X,d)$ can be isometrically embedded into $\ell_{2}^{(n)}$ as an affinely independent set.

\item[(2)] $(-1)^{k+ 1} \text{CM}\det(\{x_{0}, \ldots, x_{k}\}, d) > 0$ for all $k = 1, \ldots, n$.
\end{enumerate}
\end{theorem}

As noted, condition (1) in Theorem \ref{menger} merely says that $(X,d)$ has strict $2$-negative type.
It is also clear that we may alter the definition of the Cayley-Menger determinant by replacing the
$2$-distance matrix $D_{2}$ with the $p$-distance matrix $D_{p}$ ($0 \leq p \leq 2$). This is the same
thing as replacing $d$ with $d^{p/2}$ in the statement of Theorem \ref{menger}. It follows that we may
use Corollary \ref{cor3.5} to formulate a characterization of finite metric spaces of strict $p$-negative
type in terms of Cayley-Menger determinants ($0 \leq p \leq 2$). We omit the details.

\section{Roundness properties of additive and non-ultrametric leaf metric spaces}\label{s4}

In this section we examine roundness properties of additive metric spaces with a particular
emphasis on certain leaf metric spaces which are not ultrametric. The starting point is an elementary
statement about the roundness and generalized roundness of additive metric spaces.
This result is essentially folklore at this time but we include a short proof for completeness.

\begin{proposition}\label{addone}
Every additive metric space has generalized roundness $\geq 1$ and roundness $\geq 2$. Moreover, there exist
additive metric spaces of generalized roundness exactly $1$ and roundness exactly $2$.
\end{proposition}

\begin{proof}
Let $(X,d)$ be an additive metric space. Every finite metric subspace of $(X,d)$ is additive and therefore
embeds isometrically into a finite metric tree by Buneman's criterion (\cite[Theorem 2]{Bun}). However, it
well-known that all finite metric trees have $1$-negative type and hence generalized roundness at least $1$.
Thus the finite subspaces of $(X,d)$ are all of generalized
roundness at least $1$. This implies that $(X,d)$ is of generalized roundness at least $1$.
Similarly, every quadruple of points $x_{00}, x_{01}, x_{11}, x_{10} \in X$ embeds isometrically into a
finite metric tree by Buneman's criterion and, moreover, it is an observation of Linial and
Naor that $2$ is a roundness exponent of every metric tree. (The proof is given in Naor and Schechtman
\cite[Proposition 2]{Nao}.) Thus $(X,d)$ is of roundness at least $2$.

The complete binary tree $B_{\infty}$ of depth $\infty$ endowed with the usual path metric
has roundness at least $2$ by Proposition 2 in \cite{Nao} but it also contains segments with metric midpoints.
Thus $B_{\infty}$ has roundness exactly $2$. Moreover, $B_{\infty}$ has generalized roundness exactly $1$
by Theorem 2.2 in Caffarelli \textit{et al}.\ \cite{Caf}.
\end{proof}

\begin{corollary}\label{addtwo}
Every finite additive metric space has generalized roundness $> 1$.
\end{corollary}

\begin{proof}
Each finite metric tree $(T,d)$ has strict $q$-negative type for some $q > 1$ that depends on $|T|$ only.
(See Theorem 5.4 in Doust and Weston \cite{Dou} or Corollary 4.5 (b) in Li and Weston \cite{Hli}.)
The result now follows from Buneman's criterion and Theorem 2.4 in Lennard \textit{et al}.\ \cite{Ltw}.
\end{proof}

Proposition \ref{addone} shows that the generalized roundness of an additive metric space
must belong to $[1, \infty]$. We will see in the next section (Theorem \ref{Mathav})
that an additive metric space that is not ultrametric must have finite generalized roundness.
So it follows from Corollary \ref{addtwo} that the generalized roundness of a finite additive
metric space that is not ultrametric must belong to the interval $(1, \infty)$. Our first theorem in this
section shows that given any $\wp \in (1, \infty)$ it is possible to construct a finite leaf metric
space whose generalized roundness is exactly $\wp$. The class of metric spaces that we introduce
for this purpose is defined as follows.

\begin{definition}\label{leafy}
Let $b > 1$ be a real number and let $k \geq 2$ be an integer.
The complete bipartite graph $K_{1, k+1}$ is a star with $k+1$ leaves. We define a weighted graph
metric $d$ on $K_{1, k+1}$ by allowing $k$ edges in the star to have length $b/(b+1)$ while the remaining
edge in the star is taken to have length $1/(b+1)$. We then let $(L_{b,k}, d)$ denote the metric subspace
of $(K_{1, k+1}, d)$ that consists of the $k+1$ leaves of $K_{1, k+1}$ only. The non-zero distances in
this leaf metric space are $1$ and $z = 2b/(b+1)$.
\end{definition}

The following proposition computes the generalized roundness of the leaf metric space $(L_{b,k}, d)$.

\begin{proposition}\label{Lmet}
Let $b > 1$ be a real number and let $k \geq 2$ be an integer.
Then the leaf metric space $(L_{b,k},d)$ has generalized roundness
$$\mathfrak{q}(L_{b,k},d) = \log_{z} \left(\frac{2k}{k-1}\right),$$
where $z =\frac{2b}{b+1}$.
\end{proposition}

\begin{proof} We use a technique of S\'{a}nchez \cite{San} (Corollary 2.4) to compute generalized roundness of $(L_{b,k}, d)$.
Let $p \geq 0$ be given. An associated $p$-distance matrix for the metric space $(L_{b,k}, d)$ is given by:
$$
D_p
=
\begin{pmatrix}
0 &1 &\cdots &\cdots &1 \\
1 &\ddots &z^p &\cdots &z^p \\
\vdots &z^p &\ddots &\ddots &\vdots \\
\vdots &\vdots &\ddots &\ddots &z^p \\
1 &z^p &\cdots &z^p &0
\end{pmatrix}.
$$
We first show that the matrix $D_p$ is invertible.
Let a vector $\mathbf{a} = (a_{1}, \ldots, a_{k+1})^{T} \in \mathbb{R}^{k+1}$ such that $D_p \mathbf{a} = \mathbf{0}$ be given.
By matrix multiplication we obtain the following system of linear equations:
\begin{eqnarray*}
a_2 + a_3 + \cdots + a_{k+1} &=& 0, \\
a_1 + z^p(0 \cdot a_2 + a_3 + \cdots + a_{k+1}) &=& 0, \\
a_1 + z^p(a_2 + 0 \cdot a_3 + \cdots + a_{k+1}) &=& 0, \\
&\vdots& \\
a_1 + z^p(a_2 + a_3 + \cdots + 0 \cdot a_{k+1}) &=& 0.
\end{eqnarray*}
From all but the first equation we see that $a_2 = \cdots = a_{k+1}$.
The first equation then gives $k a_2 = 0$. As a result
$a_2 = \cdots = a_{k+1} = 0$. Then the second equation (for example) implies that
$a_1 = 0$ too. Thus $\mathbf{a} = \mathbf{0}$ and so it follows that the matrix
$D_p$ is invertible.

We now compute $\langle D_p^{-1}\mathbf{1}, \mathbf{1}\rangle = \mathbf{1}^TD_p^{-1}\mathbf{1}$,
where $\mathbf{1}$ denotes the vector whose entries are all $1$, as a function of $p$.
By setting $\mathbf{b} = D_p^{-1}\mathbf{1} = (b_{1}, \ldots, b_{k+1})^{T}$, or equivalently
$D_p\mathbf{b} = \mathbf{1}$, we may solve for $\mathbf{b}$ in the same manner as for $\mathbf{a}$
above. We obtain $b_2 = \cdots = b_{k+1} = \frac{1}{k}$
and $b_1 = 1 - z^p\left(\frac{k-1}{k}\right)$.
Hence $$\langle D_p^{-1}\mathbf{1}, \mathbf{1}\rangle
= \mathbf{1}^T \mathbf{b}
= 1 - \left(\frac{k-1}{k}\right)z^p + k \cdot \frac{1}{k} = 2 - \left(\frac{k-1}{k}\right)z^p.$$
Therefore $\langle D_p^{-1}\mathbf{1}, \mathbf{1}\rangle = 0$ if and only if $z^{p} = \frac{2k}{k-1}$.
It is now follows from Corollary 2.4 in \cite{San} that
$$\mathfrak{q}(L_{b,k},d) = \log_{z}\left(\frac{2k}{k-1}\right).$$
\end{proof}

\begin{theorem}\label{Lmet1}
For each $\wp \in (1, \infty)$ there exists a finite leaf metric space $(L,d)$ of generalized
roundness $\wp$.
\end{theorem}

\begin{proof}
Let $\wp \in (1, \infty)$ be given. Let $b$ and $k$ be as in the statement of Definition \ref{leafy}.
As $b \in (1,\infty)$, the continuous variable $z = 2b/(b+1)$ has range $(1,2)$.
Now consider the corresponding continuous function
\[
f_{k}(z) = \log_{z}\left(\frac{2k}{k-1}\right) = \frac{\log \bigl(\frac{2k}{k-1}\bigl)}{\log z}, z \in (1,2).
\]
It is easy to verify that $\text{Range} f_{k} = \bigl(\log_{2} \bigl(\frac{2k}{k-1}\bigl),
\infty \bigl)$. So, provided $k$ is sufficiently large, $\wp \in \text{Range} f_{k}$.
This means that we may choose a $z \in (1,2)$ such that $z^{\wp} = \frac{2k}{k-1}$. Moreover,
we may express $z$ in the form $2b/(b + 1)$ for an appropriately chosen
value of $b \in (1, \infty)$. The resulting finite leaf metric space $(L_{b,k}, d)$
has generalized roundness $\wp$ by Proposition \ref{Lmet}.
\end{proof}

\section{Classifications of ultrametric spaces according to roundness}\label{s5}

We begin this section with an unpublished theorem that is due to Mathav Murugan.
On the basis of Theorems \ref{embed} and \ref{Mathav} we then show how to recover and generalize certain
well-known isometric embedding theorems due to Kelly \cite{Ke1, Ke2}, Timan and Vestfrid \cite{Tim},
Lemin \cite{Lem}, and Shkarin \cite{Shk}.

\begin{theorem}\label{Mathav}
For a metric space $(X,d)$, the following statements are equivalent:
\begin{enumerate}
\item $X$ is ultrametric.
\item $X$ has infinite generalized roundness.
\item $X$ has infinite roundness.
\end{enumerate}
\end{theorem}

\begin{proof} (1) $\Rightarrow$ (2) Let $(X,d)$ be an ultrametric space and suppose $p > 0$. By Lemma \ref{c1},
$(X, d^{p})$ is an ultrametric space. In particular, $(X, d^{p})$ is an
additive metric space. So, by Proposition \ref{addone} and the interval
property of generalized roundness exponents noted in Remark \ref{int:rem}, it follows that
$1 \in gr(X, d^{p})$. Thus $p \in gr(X,d)$ by Lemma \ref{c2} (3). As $p > 0$ was arbitrary, (2) holds.

(2) $\Rightarrow$ (3) By Lemma \ref{c2} (1), we have $\mathfrak{q}(X,d) \leq \mathfrak{p}(X,d)$, so if $\mathfrak{q}(X,d) = \infty$,
we must have $\mathfrak{p}(X,d) = \infty$.

(3) $\Rightarrow$ (1) We argue contrapositively. Suppose that $(X,d)$ is not an ultrametric space.
Then there must exist points $x,y,z \in X$ such that $d(x,z) > \max \{ d(x,y), d(y,z) \}$. Let
$\alpha = d(x,z), \beta = d(x,y)$ and $\gamma = d(y,z)$. By scaling the metric $d$ by $d(x,z)^{-1}$, if necessary,
we may assume that $\alpha = 1$. This forces $0 \leq \beta, \gamma < 1$. We now
define a quadruple of points $\{ x_{ij} \}$ in $X$: $x_{00} = x, x_{11} = z$ and $x_{01} = x_{10} = y$.
If $p$ is a roundness exponent of $(X,d)$ then, by considering the indicated quadruple $\{ x_{ij} \}$, it follows that
\begin{eqnarray}\label{bound}
1 \leq 2 \cdot (\beta^{p} + \gamma^{p}).
\end{eqnarray}
However, the right side of (\ref{bound}) decreases to $0$ as $p \rightarrow \infty$, so we conclude
that $p$ is bounded away from $+ \infty$. This shows that $(X,d)$ has finite roundness.
\end{proof}

On the basis of Theorems \ref{LTWthm}, \ref{embed} and \ref{Mathav} it is possible to draw a number of very
intriguing corollaries. We conclude this paper by discussing several such corollaries.

\begin{corollary}\label{cor0}
For each $\wp \in [1, \infty]$ there exists an additive metric space whose generalized roundness is $\wp$.
\end{corollary}

\begin{proof}
This is plainly evident from Proposition \ref{addone}, Theorem \ref{Lmet1} and Theorem \ref{Mathav}.
\end{proof}

Theorem \ref{Mathav} also provides for the following additional characterizations of
ultrametric spaces.

\begin{corollary}\label{cor1}
For a metric space $(X,d)$, the following statements are equivalent:
\begin{enumerate}
\item $X$ is ultrametric.
\item $X$ has strict $p$-negative type for all $p \geq 0$.
\item $X$ admits no $p$-polygonal equality for any $p \geq 0$.
\end{enumerate}
\end{corollary}

\begin{proof}
(1) $\Rightarrow$ (2) Let $(X,d)$ be ultrametric.
By Theorem \ref{Mathav}, we have that $gr(X,d)= [0,\infty)$. So by Theorem 5.4 in Li and Weston \cite{Hli}
it follows that $(X,d)$ has strict $p$-negative type for all $p \geq 0$.

(2) $\Rightarrow$ (1) If $(X,d)$ has strict $p$-negative type for all $p \geq 0$, then
$\mathfrak{q}(X,d)=\infty$ by Theorem \ref{LTWthm}. This implies that $(X,d)$ is ultrametric by Theorem \ref{Mathav}.

(2) $\Rightarrow$ (3) Let $p \geq 0$. If $(X,d)$ has strict $p$-negative type, then the same is
true of every finite metric subspace of $X$. This entails that no $p$-polygonal equality can hold in $X$.

(3) $\Rightarrow$ (2) Suppose that for all $p \geq 0$, no $p$-polygonal equality holds in $(X,d)$.
This implies that no $p$-polygonal equality holds in any finite metric subspace of $X$. Therefore no
finite metric subspace of $X$ has finite supremal $p$-negative type by Corollary 4.4 in \cite{Hli}.
As a result, for all $p > 0$, each finite metric subspace of $X$ has strict $p$-negative type. This
implies (2). (The case $p = 0$ is trivial.)
\end{proof}

It is a deep result of metric geometry that every ultrametric space is isometric to a subset of some
Hilbert space. This result was developed independently by Kelly \cite{Ke1, Ke2}, Timan and Vestfrid \cite{Tim},
and Lemin \cite{Lem}. We will now indicate an alternate proof of this result by using Corollary
\ref{cor1} in conjunction with a classical embedding theorem of Schoenberg \cite[Theorem 1]{Sc3}.

\begin{corollary}\label{cor3}
Every ultrametric space is isometric to a subset of Hilbert space.
\end{corollary}

\begin{proof}
For a separable ultrametric space this corollary follows from Corollary \ref{cor1} (2) with $p = 2$
and Theorem 1 in Schoenberg \cite{Sc3}. In fact, Theorem 1 in \cite{Sc3} holds for a general metric
space. (The transition to the non-separable case is explained clearly in the proofs of
Lemma 2.3 and Theorem 2.4 in Wells and Williams \cite{Wel}.) Therefore the corollary holds in
full generality.
\end{proof}

It is also the case that Corollary \ref{cor1} and Theorem \ref{embed} imply the following isometric
embedding theorem that is originally due to Lemin \cite[Theorem 1.1]{Lem}.

\begin{corollary}\label{cor4}
Let $n > 1$ be an integer.
An $n+1$ point ultrametric space $(\{ x_{0}, x_{1}, \ldots, x_{n} \}, d)$ can be isometrically
embedded in the Euclidean space $\mathbb{R}^{n}$ of dimension $n$ but it cannot be isometrically
embedded in any Euclidean space $\mathbb{R}^{r}$ of dimension $r< n$.
\end{corollary}

\begin{proof}
All ultrametric spaces have strict $2$-negative type by Corollary \ref{cor1} (2) with $p = 2$.
The result now follows directly from Theorem \ref{embed}.
\end{proof}

Lemin's proof of Corollary \ref{cor4} proceeds in terms of a series of ingenious geometric lemmas and it
is quite different to the approach taken in this paper. It is worth noting that Theorem \ref{embed}
is more general than Corollary \ref{cor4}. For example, Proposition \ref{Lmet} shows how to construct finite
metric spaces $(L_{b,k},d)$ that have strict $2$-negative type (for suitably chosen $b$ and $k$)
but which are not ultrametric. Theorem \ref{embed} may also be used to generalize Lemin \cite[Theorem 1.2]{Lem}.
The idea is to consider metric spaces of infinite cardinality $\psi$ that have $p$-negative type for some $p > 2$.
For brevity, we will call such infinite metric spaces $\psi$-\textit{hard}. The class of all $\psi$-hard metric spaces
includes all ultrametric spaces of cardinality $\psi$ by Corollary \ref{cor1}. On the other hand, if we modify
Definition \ref{leafy} by letting $k = \psi$, then we obtain a non-ultrametric $\psi$-hard metric space $(L_{b,\psi}, d)$.
The supremal $p$-negative type of $(L_{b,\psi}, d)$ is $\log_{z}(2)$ where $z=2b/(b+1)$.

\begin{theorem}\label{psi:hard}
Every $\psi$-hard metric space $(X,d)$ is isometric to a closed subset of some algebraically
$\psi$-dimensional Euclidean space $E$. However, no $\psi$-hard metric space may be isometrically embedded
into any algebraically $\sigma$-dimensional Euclidean space $F$ for any $\sigma < \psi$.
\end{theorem}

The proof of Theorem \ref{psi:hard} is nearly identical to that of Theorem 1.2 in \cite{Lem}. The only changes are
that Theorem \ref{embed} is used in the place of Theorem 1.1 in \cite{Lem} and the ultrametric condition on the metric space
is replaced with $p$-negative type for some $p > 2$. It is then the case that the metric space has strict $q$-negative
type for all $q < p$ including, notably, $q = 2$. Moreover, by completing $E$, orthogonalizing the basis in it and taking
densities into account, we further see that: Every $\psi$-hard metric space of weight $\omega$ may be isometrically
embedded into a Hilbert space of weight $\omega$. The argument follows the exact same arc as that of
Lemin \cite[Theorem 1.3]{Lem}. In particular, every separable metric space that has $p$-negative type for
some $p > 2$ may be isometrically embedded into $\ell_{2}$. For separable ultrametric spaces, this result is
due to Lemin \cite{Lem} and Timan and Vestfrid \cite{Tim}.

It is worth noting that there is a Banach space generalization of Lemin \cite[Theorem 1.1]{Lem}.
Shkarin \cite{Shk} considers the class $\mathcal{M}$ of all finite metric spaces $(X,d)$,
$X = \{x_{0}, x_{1}, \ldots, x_{n} \}$ ($n > 1$), which admit an isometric embedding $\phi : X \rightarrow \ell_{2}$
such that the vectors $\{ \phi(x_{j}) - \phi(x_{0}) : 1 \leq j \leq n \}$ are linearly independent.
Theorem 1 in \cite{Shk} shows that for any $(X,d) \in \mathcal{M}$, there exists a natural
number $m = m(X,d)$ such that for any Banach space $B$ with $\dim B \geq m$, there exists an
isometric embedding of $(X,d)$ into $B$. The class $\mathcal{M}$ is not explicitly described
in \cite{Shk}, except to say that it contains all finite ultrametric spaces.
However, Theorem \ref{embed} readily implies that $\mathcal{M}$ consists of all finite metric spaces of strict
$2$-negative type.

\begin{corollary}\label{cor4.5}
Shkarin's class $\mathcal{M}$ consists of all finite metric spaces of strict $2$-negative type.
\end{corollary}

We note that Corollary \ref{cor4.5} has been obtained independently (using different techniques) by Weston \cite{We2}.

In the case of finite metric spaces there is an additional way to characterize ultrametricity.

\begin{corollary}\label{cor5}
Let $n \geq 1$. Let $(X,d) = (\{ x_{0}, x_{1}, \ldots, x_{n} \}, d)$ be a finite metric space
with associated $p$-distance matrices $D_{p} = (d^{p}(x_{i},x_{j}))$, $p \geq 0$. Then:
$(X,d)$ is ultrametric if and only if $\text{det}\, D_{p} \not= 0$ and
$\langle D_{p}^{-1} \mathbf{1}, \mathbf{1} \rangle \not= 0$ for all $p \geq 0$.
(As before, $\mathbf{1}$ denotes the vector whose entries are all $1$ and $\langle \cdot \, , \cdot \rangle$
is the standard inner-product.)
\end{corollary}

\begin{proof}
By Theorem 2.3 in \cite{San}, the condition on the $p$-distance matrices in the statement of the corollary
is equivalent to the metric space $(X,d)$ having strict $p$-negative type for all $p \geq 0$. We may
then apply Corollary \ref{cor1} to obtain the desired conclusion.
\end{proof}

It is entirely possible to view roundness as a precursor of the Banach space notion of Rademacher type.
There is a vast literature on Rademacher type and we refer the reader to the monograph Diestel \textit{et al}.\ \cite{Die}.
Indeed, if $p$ is a roundness exponent of a Banach space, then Theorem 2.1 in Enflo \cite{En1} implies that the
Banach space has Rademacher type $p$. Bourgain, Milman and Wolfson \cite{Bou} introduced a notion of
\textit{metric type} for a general metric space and showed that it is consistent with the notion of
Rademacher type for Banach spaces. Recently, Mendel and Naor \cite{Mmn} introduced a notion of \textit{scaled Enflo type}
and showed that a Banach space has scaled Enflo type $p$ if and only if it has Rademacher type $p$.
Once again, if $p$ is a roundness exponent of a metric space, then Theorem 2.1 in \cite{En1} implies that the
metric space has both metric type $p$ and scaled Enflo type $p$. In the light of these comments and Proposition
\ref{roundexp} we obtain the following corollary.

\begin{corollary}\label{scaledexp}
Let $(X,d)$ be an ultrametric space. Then $(X,d)$ has metric type $p$ and scaled Enflo type
$p$ for every $p \in [1, \infty)$.
\end{corollary}

\section*{Acknowledgments}

The research presented in this paper was initiated at the 2011 Cornell University \textit{Summer Mathematics Institute}
(SMI) and completed at the University of South Africa. The authors would like to thank the
Department of Mathematics and the Center for Applied Mathematics at Cornell University
for supporting this project, and the National Science Foundation for its financial support of
the SMI through NSF grant DMS-0739338. Additional financial support was provided by the Australian Catholic
University, Canisius College, University of New South Wales, and the University of South Africa. We are greatly indebted to a
number of our colleagues for providing highly insightful comments on our work, including the suggestion that we note
Corollary \ref{cor3.5} and give a wholly self-contained proof of Theorem \ref{Mathav}.

\bibliographystyle{amsalpha}

\end{document}